\documentclass[11pt]{amsart}
\usepackage{amsmath,amssymb,amsthm,amscd,verbatim}
\usepackage{graphicx}
\usepackage{epsfig}
\usepackage{hyperref}

\newtheorem{thm}{Theorem}
\newtheorem{lem}[thm]{Lemma}

\newtheorem{prop}[thm]{Proposition}
\newtheorem{cor}[thm]{Corollary}
\newtheorem*{thm*}{Theorem}

\theoremstyle{definition}

\theoremstyle{remark}
\newtheorem{remark}[thm]{Remark}

\newtheorem*{acknowledgement}{Acknowledgments}

 \newcommand{\cP}{\mathcal{P}}

\newcommand{\NOS}{\mathbb{N}}
\newcommand{\OS}{\mathbb{S}}

\newcommand{\bo}{\ensuremath{\mathrm{box}}}

\renewcommand{\le}{\leqslant}
\renewcommand{\leq}{\leqslant}
\renewcommand{\ge}{\geqslant}
\renewcommand{\geq}{\geqslant}

\begin{document}
\title{Boxicity and topological invariants}
\author{Louis Esperet} \address{Laboratoire G-SCOP (CNRS,
  Grenoble-INP), Grenoble, France}
\email{louis.esperet@g-scop.fr}

\thanks{Louis Esperet is partially supported by ANR Project STINT
  (\textsc{anr-13-bs02-0007}), and LabEx PERSYVAL-Lab
  (\textsc{anr-11-labx-0025}).}

\date{}
\sloppy

\begin{abstract}
  The boxicity of a graph $G=(V,E)$ is the smallest integer $k$ for
  which there exist $k$ interval graphs $G_i=(V,E_i)$, $1 \le i \le
  k$, such that $E=E_1 \cap \cdots \cap E_k$. In the first part of
  this note, we prove that every
  graph on $m$ edges has boxicity $O(\sqrt{m \log m})$, which is
  asymptotically best possible. We use this result to study the connection between
  the boxicity of graphs and their Colin de Verdi\`ere invariant,
  which share many similarities. Known results concerning the two
  parameters suggest that for any graph $G$, the boxicity of
  $G$ is at most the Colin de Verdi\`ere invariant of $G$, denoted by $\mu(G)$. We
  observe that every graph $G$ has boxicity $O(\mu(G)^4(\log \mu(G))^2)$,
  while there are graphs $G$ with boxicity $\Omega(\mu(G)\sqrt{\log
    \mu(G)})$. In the second part of this note, we focus on graphs
  embeddable on a surface of Euler genus $g$. We prove that these
  graphs have boxicity $O(\sqrt{g}\log g)$, while some of these graphs
  have boxicity $\Omega(\sqrt{g \log g})$. This improves the
  previously best known upper and lower bounds. These results directly
  imply a nearly optimal bound on the dimension of the
  adjacency poset of graphs on surfaces.
\end{abstract}
\maketitle

\section{Introduction}

Given a collection ${\mathcal C}$ of subsets of a set $\Omega$, the
\emph{intersection graph} of ${\mathcal C}$ is defined as the
graph with vertex set ${\mathcal C}$, in which two elements of
${\mathcal C}$ are adjacent if and only if their intersection is non empty.
A \emph{$d$-box} is the Cartesian product $[x_1,y_1] \times \ldots
\times [x_d,y_d]$ of $d$ closed intervals of the real line.  The
\emph{boxicity} $\bo(G)$ of a graph $G$, introduced by
Roberts~\cite{Rob69} in 1969, is the smallest integer $d
\geq 1$ such that $G$ is the intersection graph of a collection of
$d$-boxes. The \emph{intersection} $G_1 \cap \cdots \cap G_k$ of $k$ graphs
$G_1,\ldots ,G_k$ defined on the same vertex set $V$, is the graph
$(V,E_1 \cap \ldots \cap E_k)$, where $E_{i}$ ($1 \le i \le k$)
denotes the edge set of $G_{i}$.  Observe that the boxicity of a graph
$G$ can equivalently be defined as the smallest $k$ such that $G$ is
the intersection of $k$ interval graphs.

\smallskip

In the first part of this note, we prove that every graph on $m$
edges has boxicity $O(\sqrt{m \log m})$, and that there are examples
showing that this bound is asymptotically best possible.

\smallskip

A minor-monotone graph invariant, usually denoted by $\mu(\cdot)$, was
introduced by Colin de Verdi\`ere in 1990~\cite{Col90}. It relates to
the maximal multiplicity of the second largest eigenvalue of the
adjacency matrix of a graph, in which the diagonal entries can take
any value and the entries corresponding to edges can take any
positive values (a technical assumption, called the Strong Arnold
Property, has to be added to avoid degenerate cases, but we omit the
details as they are not necessary in our discussion).

\smallskip

It was proved by Colin de Verdi\`ere that $\mu(G)\le 1$ if and only if
$G$ is a linear forest, $\mu(G)\le 2$ if and only if $G$ is an
outerplanar graph, and $\mu(G)\le 3$ if and only if $G$ is a planar
graph. Scheinerman proved in 1984 that outerplanar graphs have boxicity at
most two \cite{Sch84} and Thomassen proved in 1986 that planar graphs
have boxicity at most three \cite{Tho86}. Since a linear forest is an
interval graph, these results prove that for any
planar graph $G$, $\bo(G)\le \mu(G)$.

\smallskip

These two graph invariants share several other similarities: every
graph $G$ of treewidth at most $k$ has $\bo(G)\le k+1$~\cite{CS07} and
$\mu(G)\le k+1$~\cite{GB11}. For any vertex $v$ of $G$, $\bo(G-v)\le
\bo(G)+1$ and if $G-v$ contains an edge, $\mu(G-v)\le \mu(G)+1$. Both
parameters are bounded for graphs $G$ with crossing number at most
$k$: $\bo(G)=O(k^{1/4}(\log k)^{3/4})$~\cite{ACM11} and $\mu(G)\le
k+3$~\cite{Col90}. It is known that every graph on $n$ vertices has
boxicity at most $n/2$, and equality holds only for complements of
perfect matchings~\cite{Rob69}. These graphs have Colin de Verdi\`ere
invariant at least $n-3$~\cite{KLV97}. On the other hand every graph
on $n$ vertices has Colin de Verdi\`ere invariant at most $n-1$, and
equality holds only for cliques (which have boxicity 1).

\smallskip

It is interesting to note that in each of the results above, the
known upper bound on the boxicity is better than the known upper bound
on the Colin de Verdi\`ere invariant. This suggests that for any graph
$G$, $\bo(G)\le \mu(G)$.

The following slightly weaker relationship between the boxicity and
the Colin de Verdi\`ere invariant is a direct consequence of the fact
that any graph $G$ excludes the clique on $\mu(G)+2$ vertices as a
minor, and graphs with no $K_t$-minor have boxicity $O(t^4 (\log
t)^2)$~\cite{EJ13}.

\begin{prop}\label{pro1}
There is a constant $c_0$ such that for any graph $G$, $\bo(G)\le c_0 \mu(G)^4(\log \mu(G))^2$.
\end{prop}

It follows that the boxicity is bounded by a polynomial function of the Colin de Verdi\`ere invariant.

\medskip

Pendavingh~\cite{Pen98} proved that for any graph $G$ with $m$ edges,
$\mu(G)\le \sqrt{2m}$. Interestingly, there did not exist any
corresponding result for the boxicity and it was suggested by Andr\'as
Seb\H o that graphs $G$ with large boxicity (as a function of their
number of edges) might satisfy $\bo(G)>\mu(G)$. As we observe in the
next section, there are graphs on $m$ edges, with boxicity
$\Omega(\sqrt{m \log m})$. It follows that there are graphs $G$ with
boxicity $\Omega(\mu(G)\sqrt{\log \mu(G)})$. These graphs show that
the boxicity is not even bounded by a linear function of the Colin de
Verdi\`ere invariant.

\medskip

In the second part of this paper, we show that every graph embeddable
on a surface of Euler genus $g$ has boxicity $O(\sqrt{g}\log g)$,
while there are graphs embeddable on a surface of Euler genus $g$ with
boxicity $\Omega(\sqrt{g\log g})$. This improves the upper bound
$O(g)$ and the lower bound $\Omega(\sqrt{g})$ given in~\cite{EJ13}. (Incidentally,
graphs embeddable on a surface of Euler genus $g$ have Colin de
Verdi\`ere invariant $O(g)$ and it is conjectured that the right bound
should be $O(\sqrt{g})$~\cite{Col90,Sev02}.)

Our upper bound on the boxicity of
graphs on surfaces has a direct corollary on the dimension of the
adjacency poset of graphs on surfaces, introduced by Felsner and
Trotter~\cite{FT00}, and investigated in~\cite{FLT10} and~\cite{EJ13}.

\section{Boxicity and the number of edges}\label{sec:edges}

We will use the following two lemmas of Adiga, Chandran, and
Mathew~\cite{ACM11}. A graph $G$ is \emph{$k$-degenerate} if every
subgraph of $G$ contains a vertex of degree at most $k$. In what
follows, the logarithm is taken to be the natural logarithm (and its
base is denoted by $e$).

\begin{lem}{\cite{ACM11}}\label{lem:1}
Any $k$-degenerate graph on $n\ge 2$ vertices has
boxicity at most $(k+2)\lceil 2e \log n \rceil$.
\end{lem}

\begin{lem}{\cite{ACM11}}\label{lem:2}
Let $G$ be a graph, and let $S$ be a set of vertices of $G$. Let $H$
be the graph obtained from $G$ by removing all edges between pairs of
vertices of $S$. Then $\bo(G)\le 2 \,\bo(H)+\bo(G[S])$, where $G[S]$
stands for the subgraph of $G$ induced by $S$.
\end{lem}

We now prove that every graph on $m$ vertices has boxicity
$O(\sqrt{m \log m})$. We make no real effort to optimize the
constants, and instead focus on simplifying the computation as much
as possible.

\begin{thm}\label{th:up}
For every graph $G$ on $n\ge 2$ vertices and $m$ edges, $\bo(G)\le (15e+1) \sqrt{m \log n}$.
\end{thm}

\begin{proof}
Let $G=(V,E)$ be a graph on $n\ge 2$ vertices and $m$ edges. Since the
boxicity of a graph is the maximum boxicity of its connected
components, we can assume that $G$ is connected (in particular, $m\ge
n-1\ge \log n$). Let $S$ be a set of vertices of $G$ obtained as follows: start
with $S=V$ and as long as $S$ contains a vertex $v$ with at most
$\sqrt{m/\log n}$ neighbors in $S$, remove $v$ from $S$. Let $H$ be
the graph obtained from $G$ by removing all edges between pairs of
vertices of $S$.

\smallskip

The order in which the vertices were removed from $S$ shows that the
graph $H$ is $\sqrt{m/\log n}$-degenerate. By Lemma~\ref{lem:1}, for
$k\ge 1$, every
$k$-degenerate graph on $n\ge 2$ vertices has
boxicity at most $(k+2)\lceil 2e \log n \rceil\le \tfrac{15e}2 k \log
n$. It follows that $H$ has boxicity at most $\tfrac{15e}2\sqrt{m\log n}$.

By definition of $S$, every vertex of $S$ has degree more than
$\sqrt{m/\log n}$ in $G[S]$, so $G[S]$ has at least
$\frac{|S|}2\sqrt{m/\log n}$
edge. It follows that $|S|\le 2\sqrt{m \log n}$. Since any graph on
$N$ vertices has boxicity at most $N/2$~\cite{Rob69}, $\bo(G[S])\le
\sqrt{m \log n}$.

\smallskip

By Lemma~\ref{lem:2}, $\bo(G)\le 2 \,\bo(H)+\bo(G[S])$. It follows that
$G$ has boxicity at most $15e\sqrt{m\log n}+\sqrt{m \log
  n}=(15e+1)\sqrt{m \log n}$, as desired.
\end{proof}

\begin{remark}
As proved in~\cite{ACM11}, Lemmas~\ref{lem:1}
and~\ref{lem:2} also hold if the boxicity is replaced by the \emph{cubicity} (the smallest $k$
such that $G$ is the intersection of $k$ unit-interval graphs), so the
proof can easily be adapted to show that any graph $G$ with $m$ edges
has cubicity $O(\sqrt{m \log m})$.
\end{remark}

We now observe that Theorem~\ref{th:up} is asymptotically best
possible. Let $G_n$ be a bipartite graph with $n$ vertices in each
partite set, and such that every edge between the two partite sets is
selected uniformly at random with probability $p=1/\log n$. Using
Chernoff bound, it is easy to deduce that asymptotically almost surely
(i.e., with probability tending to 1 as $n$ tends to infinity) $G_n$
has at most $2n^2/\log n$ edges. Using a nice connection between the
dimension of a poset and the boxicity of its comparability
graph~\cite{ABC10}, Adiga, Bhowmick and Chandran deduced from a result
of Erd\H os, Kierstead and Trotter~\cite{EKT91} that there is a
constant $c_1>0$ such that asymptotically almost surely, $\bo(G_n)\ge
c_1n$ (see also~\cite{ACM11}). It follows that asymptotically almost
surely, $\bo(G_n)\ge c_1\sqrt{|E(G_n)| \log n/2}$, which shows the (asymptotic)
optimality of Theorem~\ref{th:up}.

\medskip

Recall that by~\cite{Pen98}, $\mu(G_n)\le
\sqrt{2|E(G_n)|}$. This implies the following counterpart of Proposition~\ref{pro1}.

\begin{prop}
For some constant $c_1'>0$, there are infinitely many graphs $G$ with
$\bo(G)\ge c_1' \mu(G)\sqrt{\log \mu(G)}$.
\end{prop}

\begin{remark} As mentioned earlier, it was proved
in~\cite{EJ13} that graphs with no $K_t$-minor have boxicity
$O(t^4(\log t)^2)$. Since the size of a largest clique minor in $G_n$
is at most $\mu(G_n)+1$, the discussion above implies the existence of
graphs with no $K_t$-minor and with boxicity $\Omega(t\sqrt{\log t})$.
\end{remark}

\section{Boxicity and acyclic coloring of graphs on surfaces}\label{sec:surfaces}

In this paper, a {\em surface} is a non-null compact connected
2-manifold without boundary.  We refer the reader to the book by Mohar
and Thomassen~\cite{MoTh} for background on graphs on surfaces.

A surface can be orientable or non-orientable. The \emph{orientable
  surface~$\OS_h$ of genus~$h$} is obtained by adding $h\ge0$
\emph{handles} to the sphere; while the \emph{non-orientable
  surface~$\NOS_k$ of genus~$k$} is formed by adding $k\ge1$
\emph{cross-caps} to the sphere. The {\em Euler genus} of a surface
$\Sigma$ is defined as twice its genus if $\Sigma$ is orientable, and as
its non-orientable genus otherwise.

The following is a direct consequence of~\cite[Proposition 4.4.4]{MoTh}.

\begin{lem}{\cite{MoTh}}\label{lem:k3k}
If a graph $G$ embedded in a surface of Euler genus
  $g$ contains the complete bipartite graph $K_{3,k}$ as a subgraph, then $k\le 2g+2$.
\end{lem}






A coloring of the vertices of a graph is said to be \emph{acyclic} if
it is proper (any two adjacent vertices have different colors) and any two
color classes induce a forest. The following result of~\cite{EJ13} relates acyclic coloring and boxicity.

\begin{lem}{\cite{EJ13}}\label{lem:acy}
If $G$ has an acyclic coloring with $k\ge 2$ colors, then $\bo(G) \le k(k-1)$.
\end{lem}

We will also use the following recent result of Kawarabayashi and
Thomassen~\cite{KT12} (note that the constant 1000 can easily be
improved).

\begin{thm}{\cite{KT12}}\label{th:kt}
Any graph $G$ embedded in a surface of Euler genus
  $g$ contains a set $A$ of at most $1000g$ vertices such that $G-A$
has an acyclic coloring with 7 colors.
\end{thm}

Note that combining Lemma~\ref{lem:acy} and Theorem~\ref{th:kt}, it is
not difficult to derive that graphs embedded in a surface of Euler
genus $g$ have boxicity at most $500g+42$ (a linear bound with better
constants was given in~\cite{EJ13}, using completely different
arguments). We will now prove instead that their boxicity
is $O(\sqrt{g}\log g)$.

\medskip

Given a graph $G$ and a subset $A$ of vertices of $G$, the
\emph{$A$-neighborhood} of a vertex $v$ of $G-A$ is the set of
neighbors of $v$ in $A$. We are now ready to prove the main result of
this section.

\begin{thm}\label{th:surf}
There is a constant $c_2$ such that any graph embedded in a surface of Euler genus
  $g\ge 2$ has boxicity at most $c_2 \sqrt{g}\log g$.
\end{thm}

\begin{proof}
Let $G=(V,E)$ be a graph embedded in a surface of Euler genus
$g$. By Theorem~\ref{th:kt}, $G$ contains a set $A$ of at most $1000g$
vertices such that $G-A$ has an acyclic coloring with at most 7
colors. In particular, Lemma~\ref{lem:acy} implies that $G-A$ has
boxicity at most 42.

Let $H$ be the graph obtained from $G$ by deleting all edges between
pairs of vertices of $V-A$, and then identifying any two vertices of
$V-A$ having the same $A$-neighborhood. Since any two vertices of
$H-A$ have distinct $A$-neighborhoods, $H-A$ contains at most
$1+|A|+{|A| \choose 2}$ vertices having at most two neighbors in $A$. Let
 $x,y,z$ be three vertices of $A$. By Lemma~\ref{lem:k3k}, at most
$2g+2$ vertices of $H-A$ are adjacent to each of $x,y,z$. It follows
that $H-A$ contains at most $1+|A|+{|A| \choose 2}+(2g+2){|A| \choose
  3}\le 1+|A|+\frac12|A|^2+\frac16|A|^3(2g+2)\le 10^9\,g^4$
vertices. It was proved by
Heawood (see~\cite[Theorem 8.3.1]{MoTh}) that every graph embeddable
on a surface of Euler genus $g$ is
$\frac12(5+\sqrt{1+24g})$-degenerate. Consequently, $H$ (as a
subgraph of $G$) is $\frac12(5+\sqrt{1+24g})$-degenerate and by
Lemma~\ref{lem:1}, it has boxicity at
most $$\left(\frac12(5+\sqrt{1+24g})+2\right)\lceil
2e\log(10^9\,g^4+10^3\,g)\rceil\le c_3 \sqrt{g}\log g,$$ for
some constant $c_3$.

\medskip

Let $H_1$ be the graph obtained from
$H$ by adding all edges between pairs of vertices of $H-A$. Since
$H_1-A$ is a complete graph, $\bo(H_1-A)=1$. By
Lemma~\ref{lem:2}, $\bo(H_1)\le 2\,\bo(H)+\bo(H_1-A)\le 2c_3
\sqrt{g}\log g+1$.

Let $G_1$ be the graph obtained from $G$ by adding all edges between
pairs of vertices of $V-A$. It is clear that $\bo(G_1)\le \bo(H_1)$,
since any two vertices of $G_1-A$ having the same $A$-neighborhood are
adjacent and have the same neighborhood in $G_1$, so they can be
mapped to the same $d$-box in a representation of $G_1$ as an
intersection of $d$-boxes. Hence, $\bo(G_1)\le 2c_3 \sqrt{g}\log g+1$.

\medskip

Let $G_2$ be the graph obtained from $G$ by adding all edges between a
vertex of $A$ and a vertex of $V$ (i.e. $G_2$ is obtained from
$G[V-A]$ by adding $|A|$ universal vertices). Clearly, $\bo(G_2)\le
\bo(G[V-A])\le 42$. The graphs $G_1$ and $G_2$ are supergraphs of $G$,
and any non-edge of $G$ appears in $G_1$ or $G_2$, so $\bo(G)\le
\bo(G_1)+\bo(G_2)\le 2c_3 \sqrt{g}\log g+43$. It follows that there is
a constant $c_2$, such that $\bo(G)\le c_2 \sqrt{g}\log g$, as desired.
\end{proof}

Recall the probabilistic construction mentioned at the end of
Section~\ref{sec:edges}: there is a sequence of random graphs $G_n$ on
$2n$ vertices, such that asymptotically almost surely $G_n$ has at
most $2n^2/\log n$ edges and boxicity at least $c_1n$, for some
universal constant $c_1>0$. It directly follows from Euler Formula
that the Euler genus of a graph is at most its number of edges plus 2,
so asymptotically almost surely, $G_n$ has Euler genus at most $2n^2/\log
n+2$. Hence, asymptotically almost surely, $\bo(G_n)\ge c_1n\ge
c_1\sqrt{g\log n/2}\ge \frac{c_1}2 \sqrt{g \log g}$, where $g$ stands
for the Euler genus of $G_n$. Consequently, the bound of
Theorem~\ref{th:surf} is optimal up to a factor of $\sqrt{\log g}$.

\bigskip

The {\em adjacency poset} of a graph $G=(V,E)$, introduced by Felsner
and Trotter~\cite{FT00}, is the poset $(W, \leq)$ with $W=V \cup V'$,
where $V'$ is a disjoint copy of $V$, and such that $u \leq v$ if and
only if $u=v$, or $u\in V$ and $v\in V'$ and $u,v$ correspond to two
distinct vertices of $G$ which are adjacent in $G$.  The {\em
  dimension} of a poset $\cP$ is the minimum number of linear orders
whose intersection is exactly $\cP$. It was proved in~\cite{EJ13} that
for any graph $G$, the dimension of the adjacency poset of $G$ is at
most $2\, \bo(G)+\chi(G)+4$, where $\chi(G)$ is the chromatic number
of $G$. Since graphs embedded on a surface of Euler genus $g$ have
chromatic number $O(\sqrt{g})$, we obtain the following corollary of
Theorem~\ref{th:surf}, which improves the linear bound obtained
in~\cite{EJ13} and is best possible up to a logarithmic factor.

\begin{cor}
There is a constant $c_3$ such that for any graph $G$ embedded in a
surface of Euler genus $g\ge 2$, the dimension of the
adjacency poset of $G$ is at most $c_3 \sqrt{g}\log g$.
\end{cor}

\begin{acknowledgement} I thank the members of the GALOIS group, in particular Andr\'as Seb\H
o and Yves Colin de Verdi\`ere, for the interesting discussions about
the connections between the boxicity of graphs and their Colin de Verdi\`ere invariant.
\end{acknowledgement}

\end{document}